\documentclass[english,11pt]{article}
\usepackage{microtype}
\usepackage[T1]{fontenc}
\usepackage[utf8]{inputenc}
\usepackage{geometry}
\usepackage{blkarray}
\usepackage{babel}
\usepackage{amsmath}
\usepackage{amsthm}
\usepackage{amssymb}
\usepackage{setspace}
\usepackage[unicode=true,pdfusetitle,
 bookmarks=true,bookmarksnumbered=false,bookmarksopen=false,
 breaklinks=false,pdfborder={0 0 1},backref=false,colorlinks=false]
 {hyperref}

\theoremstyle{plain}
\newtheorem{thm}{\protect\theoremname}
\theoremstyle{plain}
\newtheorem{lem}[thm]{\protect\lemmaname}
\theoremstyle{plain}
\newtheorem{prop}[thm]{\protect\propositionname}
\theoremstyle{definition}
\newtheorem{defn}[thm]{\protect\definitionname}
\newtheorem{ex}[thm]{Example}
\newenvironment{lyxlist}[1]
	{\begin{list}{}
		{\settowidth{\labelwidth}{#1}
		 \setlength{\leftmargin}{\labelwidth}
		 \addtolength{\leftmargin}{\labelsep}
		 }}
	{\end{list}}
\theoremstyle{plain}

\usepackage{babel}
\usepackage{tikz}
\providecommand{\corollaryname}{Corollary}
\providecommand{\definitionname}{Definition}
\providecommand{\lemmaname}{Lemma}
\providecommand{\propositionname}{Proposition}
\providecommand{\theoremname}{Theorem}

\newcommand{\FF}{\mathbb F}
\newcommand{\PP}{\mathbb P}
\newcommand{\QQ}{\mathbb Q}

\newcommand{\ZZ}{\mathbb Z}

\providecommand{\corollaryname}{Corollary}
\providecommand{\definitionname}{Definition}
\providecommand{\lemmaname}{Lemma}
\providecommand{\propositionname}{Proposition}
\providecommand{\theoremname}{Theorem}

\begin{document}
\title{\large \textbf{CHARACTERISTIC SETS OF MATROIDS}}
\author{\small DUSTIN CARTWRIGHT AND DONY VARGHESE}
\date{}
\maketitle
\begin{abstract}
We investigate possible linear, algebraic, and Frobenius flock characteristic
sets of matroids. In particular, we classify possible combinations of linear and
algebraic characteristic sets when the algebraic characteristic set is finite or
cofinite. We also show that the natural density of an algebraic characteristic set in the set of primes may be arbitrarily close to any real number in the unit interval.

Frobenius flock realizations can be constructed from algebraic realizations, but
the converse is not true. We show that the algebraic characteristic set may be
an arbitrary cofinite set even for matroids whose Frobenius flock
characteristic set is the set of all primes. In addition, we construct Frobenius
flock realizations in all positive characteristics from linear realizations in
characteristic~0, and also from Frobenius flock realizations of the dual
matroid.
\end{abstract}

\section{Introduction}

A matroid is a combinatorial structure generalizing
the concept of linear independence of vectors
in a vector space \cite{Whitney1935}. Explicitly, it is equivalent to the
collection of all linearly independent subsets of a fixed set of vectors in a
vector space. However, not all matroids have
linear representations, and the existence of a linear representation
can depend on the field. The linear characteristic set
of a matroid $M$, which we denote $\chi_L(M)$, is the set of characteristics of
those fields over which $M$ does have a linear representation. The linear
characteristic set of
a matroid is either a finite set of positive primes or a cofinite
set containing~$0$ \cite{Rado1957,Vamos1971}. Conversely, any
such set occurs as the linear characteristic set of some matroid~\cite{Kahn1982,reid}.

Similar to the linear independence in a vector space, algebraic independence
in a field extension also defines a matroid. For a matroid $M$ on
a set $E$, an algebraic representation over $K$ is a pair $\left(L,\phi\right)$
consisting of a field extension $L$ of $K$ and a map $\phi\colon E\rightarrow L$
such that any $I\subseteq E$ is independent in $M$ if and only if
the set $\phi(I)$ is algebraically independent over $K$. If a matroid
has a linear representation over a field $K$, then it also has an
algebraic representation over $K$. Conversely, an algebraic representation
over a field of characteristic~$0$ can be turned into a linear representation
over a field of characteristic~$0$ by using derivations. However, there are
matroids with algebraic representations in positive characteristic, but not
linear representations~\cite{Lindstrom1986}.

The algebraic characteristic set of a matroid $M$, denoted by $\chi_{A}(M)$, is
the set of characteristics of the fields in which a matroid $M$ has algebraic
representations. A classification analogous to that for linear characteristic
sets is not known. Nonetheless, our first result shows that the algebraic
characteristic set can be an essentially arbitrary finite or cofinite set, with
a restriction only for characteristic 0:
\begin{thm}
\label{thm:char-sets} Let $C_{L}\subseteq C_{A}\subseteq\mathbb{P}\cup\left\{ 0\right\} $
be finite or cofinite subsets. Suppose either that $0\in C_{L}$
and $C_{L}$ is cofinite, or that $0\notin C_{A}$ and $C_{L}$
is finite. Then there exists a matroid $M$ such that $\chi_{L}(M)=C_{L}$ and $\chi_{A}(M)=C_{A}$.
\end{thm}
\noindent Here, and throughout the paper, $\mathbb{P}$ denotes the set of all primes.

Unlike the linear characteristic set, the algebraic characteristic set of a
matroid may be neither finite nor
cofinite~\cite[Ex.~2]{EvansHrushovski}. We extend this example to show
that the possible densities of the algebraic characteristic set are dense in the
interval $[0,1]$:

\begin{thm}
\label{thm:char-density}
Let $0 \leq \alpha \leq 1$ be a real number and
$\epsilon > 0$. Then there exists a matroid $M$ such that $\lvert
d(\chi_{A}(M))-\alpha\rvert <\epsilon$, where $d(\chi_A(M))$ refers to the
natural density of $\chi_A(M)$ in the set of all primes.
\end{thm}

\noindent Recall that the natural density of a set of primes is defined as:
\[
d(S) = \lim_{N \rightarrow\infty} \frac{\lvert \{p \in S \mid p < N\} \rvert}{\lvert \{p
\in \PP \mid p < N\} \rvert},
\]
if that limit exists.

While derivations of algebraic representations in positive characteristic do not
always give linear representations of the same matroid, Lindström found
cases where they did and used that to prove that for $p$ a prime, the so-called
Lazarson matroids $M_{p}$ have algebraic characteristic set consisting of
just~$p$~\cite{Lindstrom1985}. Gordon extended this technique to give examples
of matroids with some special non-singleton finite algebraic characteristic
sets~\cite{Gordon}. He even went so far as to speculate that matroids with
non-empty finite linear characteristic set had finite algebraic characteristic
set, which is false by Theorem~\ref{thm:char-sets}.

Inspired by Lindström's work, Bollen, Draisma, and Pendavingh constructed what
they called a Frobenius flock from an algebraic realization of a matroid. The
Frobenius flock of an algebraic realization is a collection of compatible linear
realizations of different matroids, which collectively agree with the algebraic
matroid~\cite{BollenDraismaPendavingh}. The examples of Lindström and Gordon
corresponded to the case where the flock was determined by a single one of these
linear realizations. Given the tight connection between the Frobenius flock
realizations and algebraic realizations in these cases, it is natural to wonder
how different they are in general. We can define the flock characteristic set
$\chi_{F}(M) \subset \mathbb{P}$, analogously to the linear and algebraic
characteristic sets, and the flock characteristic set is an upper bound on the
algebraic characteristic set. However, we show that their difference can be an
arbitrary finite set of
primes:
\begin{thm}\label{thm:char-sets-flock}
In Theorem~\ref{thm:char-sets}, the matroids constructed with infinite
algebraic characteristic set also have $\chi_F(M) = \mathbb P$.
\end{thm}

It would be interesting to know if the flock characteristic set can be an
arbitrary cofinite set, like the linear and algebraic characteristic sets can
be. However, the combinations of flock characteristic set and linear
characteristic set are
constrained by the following:
\begin{thm}
\label{thm:frob-flock-all}Let $M$ be a matroid. If $0\in\chi_{L}(M)$,
then $\chi_{F}(M)=\mathbb{P}$. 
\end{thm}
\noindent Theorem \ref{thm:frob-flock-all}, together with results quoted above, show that
Theorem~\ref{thm:char-sets} constructs all possible triples of linear,
algebraic, and flock characteristic sets, in the case where the linear
characteristic set includes~$0$.

The method for proving Theorem~\ref{thm:frob-flock-all} involves ``stretching''
linear flocks (which are Frobenius flocks, but with a possible different
automorphism than Frobenius). A consequence of this construction, is the
following, which
disproves \cite[Conj.~8.21]{Bollen}:
\begin{thm}\label{thm:frob-dual}
If $M$ is a matroid, and $M^*$ is its dual matroid, then $\chi_F(M^*) =
\chi_F(M)$.
\end{thm}

While we don't know about cofinite flock characteristic sets, any single prime
may be a flock characteristic set~\cite{Lindstrom1985,BollenDraismaPendavingh}.
Moreover, we show that certain finite sets are also possible:
\begin{thm}
\label{thm:finite-flock-char}Let $C$ be any Gordon-Brylawski set of primes. Then
there exists a matroid~$M$ with $\chi_L(M) = \chi_A(M) = \chi_{F}(M)=C$.
\end{thm}
\noindent Gordon-Brylawski sets are sets of primes satisfying a certain
technical condition, given in Definition~\ref{def:bry} below. Although we don't
know if the cardinality of a Gordon-Brylawski set is bounded,
Example~\ref{ex:gb-80} shows that a Gordon-Brylawski set may have as many as 80
elements.

Recently, matroids over hyperfields and tracts have attracted much attention, as
a generalization of both matroids and matroid realizations over a
field~\cite{BakerBowler,Su}. In general, hyperfields and tracts don't seem to
have a natural notion of characteristic, and so it's not clear what would be the
analogous definition of a characteristic set. Nonetheless, Frobenius flock
realizations are equivalent to realizations over a certain skew hyperfield
introduced in~\cite{Pendavingh}, and so our results do cover characteristic sets
for this specific class of hyperfields. Further investigation of the
characteristic sets of Frobenius flocks could give some insight into what kind
of behavior we should expect for realizations over tracts and skew hyperfields
in general.

The remainder of this paper is organized as follows. In
Section~\ref{sec:Specified-characteristic-sets}, we construct matroids to
prove Theorem~\ref{thm:char-sets}. In
Section~\ref{sec:Infinite-algebraic-char}, we construct matroids whose algebraic
characteristic set is neither finite nor cofinite, and prove
Theorem~\ref{thm:char-density}. In Section
\ref{sec:Constructing-Frobenius-flocks}, we recall the definition of linear
flocks from \cite{BollenDraismaPendavingh} and prove the Theorem
\ref{thm:frob-flock-all}.  Finally in
Section~\ref{sec:Brylawski-matroids}, we examine examples of matroids with
finite flock characteristic sets and prove Theorem~\ref{thm:finite-flock-char}.

\section{Specified characteristic sets} \label{sec:Specified-characteristic-sets}
In this section, we introduce a lemma of Evans and Hrushovski and use it to
construct matroids with specified linear and algebraic characteristic sets.
Evans and Hrushovski constructed algebraic realizations of matroids using
matrices of endomorphisms of a fixed one-dimensional group. Moreover, they
showed that for certain matroids, all algebraic realizations are equivalent to
realizations by such matrices.

This one-dimensional group construction simultaneously generalizes the
realization of linear matroids as algebraic matroids and the realization of
rational matroids as algebraic matroids over any field using monomials. The
important point for us is that it depends on a choice of one-dimensional
connected algebraic group $G$ over an algebraically closed field~$K$. Such
groups can be classified as either $G_a$, the additive group of $K$, $G_m$, the
multiplicative group of $K$, or $E$ an elliptic curve
over~$K$~\cite[Sec.~2.1]{BollenCartwrightDraisma}. In each of these
cases the ring
of endomorphisms of the algebraic group is an integral domain $\mathbb E$, which
can be shown to be contained in a (possibly non-commutative) division ring~$D$.
The one-dimensional group construction turns a linear representation of a
matroid over $D$ into an algebraic representation over~$K$. The standard
translation of linear matroids into algebraic matroids corresponds to the
group~$G_a$, with $k \in K$ corresponding to the function $x \mapsto kx$, which
is an endomorphism of $G_a$. Likewise, the endomorphisms of the multiplicative
group $G_m$ are just the integers with $n$ corresponding to the multiplicative
endomorphism $x \mapsto x^n$, and then the group construction translates an
integer matrix into monomials.

\begin{lem}[Lem.~3.4.1 in \cite{EvansHrushovski}] \label{lem:evans-hur} Let $\Phi$ be a collection of equations in
the variables $x_{0},...,x_{n}$ including the equations and inequalities: 
\[
x_{0}=0,\;x_{1}=1, \mbox{ and } x_{i}\neq x_{j}\,(\mbox{for all }\,i\neq j),
\]
together with some equations of the form: 
\[
x_{i}=x_{j}+x_{k}\,(\mbox{where }\,j,k\neq0,k\neq i\neq j),x_{i}=x_{j}\cdot x_{k}\,(\mbox{where }\,i,j,k\neq0,1\,and\,k\neq i\neq j).
\]

Then there exists a matroid $M$ satisfying the following properties:
\begin{enumerate}
    \item $M$ has a linear realization over an infinite field~$K$ if and only if there exist (distinct) values for $x_{0},\ldots,x_{n}$ in~$K$ which simultaneously satisfy every equation in~$\Phi$.
    \item  $M$ has an algebraic realization over a field $K$ if and only if there exists a linear representation of $M$ over the division ring generated by the ring of endomorphisms of a 1-dimensional algebraic group $G$ over a field of the same characteristic as~$K$.
\end{enumerate}

\end{lem}

From now on, we will refer to systems of equations satisfying the conditions of
Lemma~\ref{lem:evans-hur} to mean the form in the first paragraph. Note that, because of the required inequalities, a solution to such a system in a division ring $Q$ will always mean an assignment of distinct values of~$Q$ for the variables.

We now recall the classification of the endomorphism rings of a one-dimensional
algebraic group. If characteristic of $K$ is $0$, then the endomorphism ring of
$G_a$, $G_m$, or an elliptic curve is, respectively, $K$, $\mathbb{Z}$, and
either $\mathbb{Z}$ or an order in an imaginary quadratic number
field~\cite[Thm.~VI.6.1(b)]{Silverman}. If
characteristic of $K$ is $p>0$, then the endomorphism ring of of $G_{m}$ is
again $\mathbb{Z}$, but the endomorphisms of $G_{a}$ are instead isomorphic to
the non-commutative ring of $p$-polynomials, denoted
$K[F]$~\cite[Sec.~20.3]{Humphreys}. Elements of $K[F]$
are written as polynomials in an indeterminate $F$, with coefficients in $K$,
but with the multiplication rule defined by $Fa = a^pF$, if $a \in K$. In
addition to the same possibilities as characteristic~$0$, the endomorphism ring
of an elliptic curve in positive characteristic may be an order in a quaternion
ring~\cite[Cor. III.9.4]{Silverman}. 
\begin{lem}
\label{lem:basic-system} Let $n>1$ be an integer. Then there exists
a system of equations $\Phi_{n}$, satisfying the conditions in Lemma~\ref{lem:evans-hur},
whose variables include $y_{i}$ for $1\leq i\leq n+1$, and $w$,
with the following properties: 
\begin{enumerate}
    \item For any solution in a division ring~$Q$ to the system of equations~$\Phi_{n}$, the variables satisfy $y_{i}=y_1^{i}$ for $2 \leq i \leq n+1$ and $w=y_1^{n+1}+ny_1^{n-1}+(n-1)y_1^{n-2}$ and the inequality $y_1^{n-1} + y_1^{n-2} \neq 0$.
    \item For any field~$K$, there exists a finite set $S \subset K$ of elements algebraic over the prime subfield of~$K$, such that for any $t \in K \setminus S$, there exists a solution to $\Phi_{n}$ with $y_{1}=t$. 
\end{enumerate}
\end{lem}

\begin{proof}
We define the system~$\Phi_{n}$ using variables denoted $x_{0},x_{1},y_{1},\ldots,y_{n+1}$,$z_{1},\ldots,z_{n-1}$,  
$w_{1},\ldots,w_{2n-3},w$, and satisfying the following equations:
\begin{align*}
 x_{0} & =0 & y_{2} & = y_{1}\cdot y_{1} & z_{1} & = y_{1}+x_{1} & w_{1} & = y_{3}+y_{1} \\
 x_{1} & =1 & y_{3} & = y_{2}\cdot y_{1} &  z_{2} & = z_{1}\cdot y_{1} & w_{2} & = w_{1}+z_{1} \\
  &  &  & \;\;\vdots  &  z_{3} & = z_{2}\cdot y_{1} &  w_{3} & = w_{2}\cdot y_{1} \\
  &  &  y_{n} & = y_{n-1}\cdot y_{1} &  &  \;\;\vdots &  w_{4} & = w_{3}+z_{2} \\
  &  &  y_{n+1} & = y_{n}\cdot y_{1} &  z_{n-1} & = z_{n-2}\cdot y_{1} & w_{5} & = w_{4}\cdot y_{1}\\
  % &  &  &  &  &  &  w_{6} & = w_{5}+z_{3} \\
  &  &  &  &  &  &  &\;\;\vdots \\
  &  &  &  &  &  &  w_{2n-3} & = w_{2n-2}\cdot y_{1} \\
  &  &  &  &  &  &  w & = w_{2n-3}+z_{n-1} 
\end{align*}

If we let $t$ denote the value of $y_{1}$, then we can recursively evaluate the
variables in terms of $t$:

\begin{align*}
x_{0} & =0 & y_{1} & = t & z_{1} & = t+1 & w_{1} & = t^{3}+t \\
x_{1} & =1 &  y_{2} & = t^{2} & z_{2} & = t\left(t+1\right) &  w_{2} & =
t^{3}+2t+1 \\ 
 &  & y_{3} & = t^{3} & z_{3} & = t^{2}\left(t+1\right) & w_{3} & = t^{4}+2t^{2}+t \\
 &  & \;\;\vdots &  &  & \;\;\vdots &  w_{4} & = t^{4}+3t^{2}+2t \\
 &  & y_{n} & = t^{n} &  z_{n-1} & = t^{n-2}\left(t+1\right) &  w_{5} & = t^{5}+3t^{3}+2t^{2}\\
 &  & y_{n+1} & = t^{n+1} &  &  & &  \;\;\vdots \\
 &  &  &  &  &  &  w_{2n-3} & = t^{n+1}+\left(n-1\right)t^{n-1} +(n-2)t^{n-2} \\
 &  &  &  &  &  &  w& = t^{n+1}+nt^{n-1}+\left(n-1\right)t^{n-2}\\
\end{align*}

This proves the first claim.

For the second claim, we need to show that there exists a finite set~$S$ in $K$,
such that any element~$t$ in $K \setminus S$
provides a solution to $\Phi_{n}$. To show
that, let us consider the above solution to $\Phi_{n}$ as polynomials in $t$
using the variable assignments as above and
let $P$ be the set of those polynomials. Let $S$ consist of all
roots of equations of the form $p-q=0$, for all distinct $p,q\in P$. Now, we
need to show that for all $p,q\in P$, $p-q$ are non-zero polynomials,
independent of the characteristic, to know that $S$ is finite. All the
polynomials in $P$ are monic, and the difference between monic polynomials of
different degrees is non-zero,
so it sufficient to show that $p-q \neq 0$ for polynomials~$p$
and~$q$ of the same degree.

The polynomials with degree $1$ are $y_{1}$ and $z_{1}$. The difference
between $y_{1}$ and $z_{1}$ is $1$, so they are distinct. Similarly,
degree $2$ elements are $y_{2}$ and $z_{2}$, their difference is
$t$ so they are distinct. For $3\leq i\leq n$, the elements with
degree $i$ are $t^{i},t^{i}+t^{i-1},t^{i}+\left(i-2\right)t^{i-2}+(i-3)t^{i-3}$
and $t^{i}+\left(i-1\right)t^{i-2}+(i-2)t^{i-3}$. The difference
between these terms are either a monic polynomial, $\left(i-2\right)t^{i-2}+(i-3)t^{i-3}$
or $\left(i-1\right)t^{i-2}+(i-2)t^{i-3}$. The terms $\left(i-2\right)t^{i-2}+(i-3)t^{i-3}$
and $\left(i-1\right)t^{i-2}+(i-2)t^{i-3}$ are not zero because a
prime cannot divide consecutive integers. So degree $i$ elements
are distinct for $3\leq i\leq n$. The elements with degree $n+1$
are $y_{n+1},w_{2n-3},$ and $w_{2n-2}$. The difference between these
terms are either a monic polynomial, $\left(n-1\right)t^{n-1}+\left(n-2\right)t^{n-2}$,
or $nt^{n-1}+\left(n-1\right)t^{n-2}$. These are not zero since because, again,
a prime cannot divide consecutive integers. Thus, $S$ is a finite set of elements, all algebraic over the
prime subfield of~$F$ and for any $t$ outside of~$S$,
each of the variables in the solution to $\Phi_{n}$ with $y_{1}=t$
will be distinct.
\end{proof}

We now use Lemma~\ref{lem:basic-system}, together with additional equations in
order to construct matroids with specified characteristic sets.

\begin{prop}
\label{prop:finite-finite} Let $C$ be a finite set of primes. Then
there exists a matroid $M$ such that $\chi_{L}(M)=\chi_{A}(M)=C$. 
\end{prop}

\begin{proof}
Let $n$ be the product of the primes in $C$. We use the system $\Phi_{n}$
from Lemma~\ref{lem:basic-system} and add the equation $y_{n+1}=w+y_{n-2}$.
Now, use Lemma~\ref{lem:evans-hur} to construct a matroid~$M$.
If $\Phi_{n}$ has a solution in a division ring~$Q$, then by
Lemma~\ref{lem:basic-system}, the two sides of our added equation evaluate to
$y_{n+1} = y_1^{n+1}$ and $w+y_{n-2} = y_1^{n+1} + ny_1^{n-1} + ny^{n-2}$, with
$y_1^{n-1} + y_1^{n-2}$ non-zero in~$Q$. Therefore, for the equation to hold,
$n$ must be $0$, which means the characteristic of~$Q$ is contained in~$C$. In
other words, $\chi_L(M) \subset C$. Also,
since the endomorphism ring of a 1-dimensional group can only have
positive characteristic if the field of definition has the same characteristic,
then $\chi_{A}(M)\subset C$.

On the other hand, for any infinite field $K$ whose characteristic
is contained in $C$, we can choose $t\in K$ outside a finite set
and have a solution to $\Phi_{n}$ with $y_{i}=t^{i}$. Furthermore,
because $n=0$ in $K$, this will also be a solution with the additional
equation, showing that $\chi_{L}(M)\supset C$ and completing the
proof of the proposition. 
\end{proof}

\begin{lem}
\label{lem:cofinite-eqns} Let $C$ be the union of $\{0\}$ and a
cofinite set of primes. Then there exists a set of equations $\Phi_{C}$,
satisfying the set of constraints in Lemma~\ref{lem:evans-hur} such
that if $\Phi_{C}$ has a solution over a division ring~$Q$, then the
characteristic of~$Q$ is contained in~$C$. Conversely
if $K$ is any infinite field whose characteristic is contained in
$C$, then $\Phi_{C}$ has a solution in~$K$. 
\end{lem}

\begin{proof}
Let $n$ be the product of the finite set of primes not in~$C$
and consider $\Phi_{n}$ from Lemma~\ref{lem:basic-system}. We will
construct a system of equations $\Phi_{C}$, by adding a variable
$v$ and the equation $v=w+y_{n-2}$ to $\Phi_{n}$.

If $Q$ is a division ring of characteristic not in $C$, then by
Lemma~\ref{lem:basic-system}, for any solution in~$Q$,
$y_{n-2}=y_1^{n-2}$ and $w=y_1^{n+1}-y_1^{n-2}$, where we're using the fact
that $n=0$ in~$Q$. By the added equation, then, $v=y_1^{n+1}$ and therefore $v =
y_{n+1}$, so the variables are distinct. We conclude that $\Phi_{C}$ does not
have a solution with distinct values over a division ring with characteristic
not in~$C$.

Conversely, if $K$ is an infinite field whose characteristic is in~$C$, then we
can choose $t$ outside of the finite set $S$ from Lemma~\ref{lem:basic-system}
to get a solution with $y_i = t^i$ and $v = t^{n+1} + n t^{n-1} + nt^{n-2}$. We
only need to justify that $v$ does not coincide with any of the variables used
in $\Phi_n$. First, as a polynomial in~$t$, it has a different degree than all
except $y_{n+1}$, $w_{2n-3}$, and $w_{2n-2}$. The differences between $v$ and
each of these are polynomials with leading coefficients $n$, $1$, and $1$
respectively, so they are distinct elements of $K$, because $n$ is non-zero in
$K$. We conclude that $\Phi_C$ has a solution in~$K$.
\end{proof}

\begin{prop}
\label{prop:cofinite-all} Let $C$ be the union of $\{0\}$ and a
cofinite set of primes. Then there exists a matroid $M$ such that
$\chi_{L}(M)=C$ and $\chi_{A}(M)=\{0\}\cup\mathbb{P}$. 
\end{prop}

\begin{proof}
Let $\Phi_{C}$ be the system of equations from Lemma~\ref{lem:cofinite-eqns}.
Then, use Lemma~\ref{lem:evans-hur} to construct a matroid $M$.
By these two lemmas, $M$ is realizable over any infinite field of
characteristic contained in $C$ and not realizable over any field
of characteristic not contained in $C$. Therefore, $\chi_{L}(M)=C$.
In particular, $M$ is realizable over $\mathbb{Q}$, which is the
field of fractions of the endomorphism ring of $G_{m}$, so $M$ is
algebraically realizable over any field. 
\end{proof}
\begin{prop}
\label{prop:cofinite-cofinite}Let $C$ be the union of $\{0\}$ and
a cofinite set of primes. Then there exists a matroid $M$ with $\chi_{L}(M)=\chi_{A}(M)=C$. 
\end{prop}

\begin{proof}
We start with the system $\Phi_{C}$ as in Lemma~\ref{lem:cofinite-eqns},
to which we add the variables $u_{1}$, $u_{2}$, and $u_{3}$ and
the equations $u_{2}=u_{1}\cdot u_{1}$, $u_{3}=u_{2}\cdot u_{1}$,
and $1 = x_{1}=u_{3}+u_{1}$ to get $\Phi$. Let $M$ be a matroid constructed
from this system according to Lemma~\ref{lem:evans-hur}. Any solution
to $\Phi$ in a division ring of characteristic 0 must satisfy
$u_{1}^{3}+u_{1}-1=0$. This polynomial is irreducible in $\mathbb{Q}$,
so the value $u_{1}$ takes must be degree three over $\mathbb{Q}$.
However, the ring of endomorphisms of $G_{m}$ or an elliptic curve
is contained in either the rationals, a quadratic number field, or
a quaternion algebra over $\mathbb{Q}$, and all elements of these
rings have degree at most~2 over $\mathbb{Q}$. Therefore, any algebraic
realization of $M$ must come from the algebraic group $G_{a}$, whose
endomorphism ring has the same characteristic as the field of definition.
Then, by Lemma~\ref{lem:cofinite-eqns}, the characteristic of any
division ring having solutions to $\Phi$, and thus to $\Phi_{C}$
must be contained in $C$, and thus $\chi_{A}(M)\subset C$.

On the other hand, we want to show that $\chi_{L}(M)\supset C$. Let $K$ be a
transcendental extension of an algebraically closed field whose characteristic
is contained in~$C$. Let $u_{1}$ be any root of the polynomial
$u_{1}^{3}+u_{1}-1$ so long as $u_{1}\neq-1$ (which is only possible in
characteristic~$3$, and in characteristic~$3$ there are also other roots). Then,
set $u_{2}=u_{1}^{2}$, and $u_{3}=u_{1}^{3}$, and we claim that $0$, $1$,
$u_{1}$, $u_{2}$, and~$u_{3}$ are distinct. We consider the possible equalities:
First, if $u_{1}$, $u_{2}$, or~$u_{3}$ is zero, then $u_{1}=0$, which is not
possible because the polynomial has a non-zero constant term. Second, if $u_1 =
1$, $u_2 = u_1$, or $u_3 = u_2$, then that implies $u_1=1$, but the defining
polynomial for $u_1$ evaluates to $1 \neq 0$ in all characteristics at $u_1 =
1$. Third, if $u_2 = 1$ or $u_3 = u_1$, then that implies $u_1 = \pm 1$, and
we've assumed that $u_1 \neq -1$ and shown that $u_1 = 1$ is not possible.
Fourth, if $u_3 = 1$, then substituting $u_3 = u_1^3$ into the defining
polynomial yields $u_1 = 0$, which is a contradiction.

Now choose $t$ to be transcendental over the prime field of $K$. Then it is not
a root of $u_1^3 + u_1 -1$ and not equal to $u_2$ or $u_3$, either. Furthermore,
all of the variables defined in $\Phi_C$ are polynomials of $t$, and thus
transcendental over the prime subfield of~$K$, and so distinct from the $u_i$'s.
This shows that all variables are distinct in this solution, and therefore
$C\subset\chi_{L}(M)$, which completes
the proof of the proposition. 
\end{proof}
\begin{prop}
\label{prop:finite-all} Let $C$ be a finite set of primes. Then
there exists a matroid $M$ with $\chi_{L}(M)=C$ and $\chi_{A}(M)=\chi_{F}(M_{P})=\mathbb{P}$. 
\end{prop}

\begin{proof}
Let $n$ be the product of the primes in $C$. Consider the system
of equations $\Phi$ consisting of $\Phi_{n}$ from Lemma~\ref{lem:basic-system}
together with additional variables $u_{1},\ldots,u_{8}$ and the equations:
\begin{align*}
u_{3} & =u_{2}+x_{1}\\
u_{4} & =u_{1}\cdot u_{3}\\
u_{5} & =u_{2}\cdot u_{1}\\
u_{6} & =u_{5}+w\\
u_{7} & =u_{6}+y_{n-2}\\
u_{8} & =u_{4}+y_{n+1}\\
u_{8} & =u_{7}+u_{1}
\end{align*}
Let $M$ be the matroid defined from $\Phi$ by Lemma~\ref{lem:evans-hur}.
If we have any solution to $\Phi$ in a division ring~$Q$, then there
exists $t\in Q$ such that $y_{i}=t^{i}$ and $w=t^{n+1}+nt^{n-2}+(n-1)t^{n-2}$
by Lemma~\ref{lem:basic-system}. If we let $a$ and $b$ be the
values of $u_{1}$ and $u_{2}$, respectively. Then, the other variables
satisfy:
\begin{align*}
u_{3} & =b+1\\
u_{4} & =ab+a\\
u_{5} & =ba\\
u_{6} & =ba+t^{n+1}+nt^{n-1}+(n-1)t^{n-2}\\
u_{7} & =ba+t^{n+1}+nt^{n-1}+nt^{n-2}\\
u_{8} & =ab+a+t^{n+1}\\
 & =ba+a+t^{n+1}+nt^{n-1}+nt^{n-2}
\end{align*}
If $Q$ is commutative, then $ab=ba$ and so the last equation implies
that $nt^{n-1}+nt^{n-2}=0$. Since $t^{n-1}+t^{n-2}$ is non-zero
by Lemma~\ref{lem:basic-system}, then $n=0$, which means that the
characteristic of a commutative field which has solutions to $\Phi$
must be contained in $C$. In particular, $\chi_L(M) \subset C$.

Conversely, let $K=\mathbb{F}_{p}(a,b,t)$, where $p\in C$ and consider
the solution formed by setting $y_{i}=t^{i}$, $u_{1}=a$, $u_{2}=b$,
and assigning the other variables as above. Then the variables $u_{1},\ldots,u_{8}$
are distinct polynomials. Moreover, the variables $u_{i}$ are not
contained in $\mathbb{F}_{p}(t)$, whereas all the variables used
by the system $\Phi_{n}$ are contained in $\mathbb{F}_{p}(t)$, so
these are also distinct. This shows that $\chi_L(M) = C$.

Finally, we want to show that $M$ is algebraically realizable over
the field $\overline{\mathbb{F}}_{p}$ for any prime~$p$. Since
$M$ is linearly realizable when $p\in C$, it is sufficient to consider
the case when $p\not\in C$, so $n$ is non-zero. We construct an algebraic
realization by finding a solution to $\Phi$ over the ring
$\overline{\mathbb{F}}_{p}[F]$, which is the endomorphism ring
of~$G_{a}$. We first choose $\alpha\in\overline{\mathbb{F}}_{p}\setminus\mathbb{F}_{p^{n-1}}\setminus\mathbb{F}_{p^{n-2}}$.
Thus, $\alpha^{p^{n-1}}-\alpha$ and $\alpha^{p^{n-2}}-\alpha$ are
non-zero, so we set $\beta=(\alpha^{p^{n-1}}-\alpha)^{-1}$ and $\gamma=(\alpha^{p^{n-2}}-\alpha)^{-1}$.
Then, let $y_{1}=F$, $u_{1}=\beta F^{n-1}+\gamma F^{n-2}$, $u_{2}=n\alpha$,
and the other variables as: 
\begin{align*}
u_{3} & =n\alpha+1\\
u_{4} & =(n\beta\alpha^{p^{n-1}}+\beta)F^{n-1}+(n\gamma\alpha^{p^{n-2}}+\gamma)F^{n-2}\\
 & =(n\alpha\beta+n+\beta)F^{n-1}+(n\alpha\gamma+n+\gamma)F^{n-2}\\
u_{5} & =n\alpha\beta F^{n-1}+n\alpha\gamma F^{n-2}\\
u_{6} & =F^{n+1}+(n\alpha\beta+n)F^{n-1}+(n\alpha\gamma+n-1)F^{n-2}\\
u_{7} & =F^{n+1}+(n\alpha\beta+n)F^{n-1}+(n\alpha\gamma+n)F^{n-2}\\
u_{8} & =F^{n+1}+(n\alpha\beta+n+\beta)F^{n-1}+(n\alpha\gamma+n+\gamma)F^{n-2}
\end{align*}
All of these are distinct values in $\overline{\mathbb{F}}_{p}(F)$
and satisfy the equations in $\Phi$. Moreover, they are distinct
from the variables used in $\Phi_{n}$, because those all lie in the
subfield $\mathbb{F}_{p}(F)$.  We conclude that $\chi_A(M) = \mathbb P$.
\end{proof}

\begin{proof}[\textbf{Proof of Theorems \ref{thm:char-sets}
and~\ref{thm:char-sets-flock}}]
We consider the different combinations of whether $C_A$ and~$C_L$ are finite or
cofinite.
First, suppose that $C_{A}$ is finite, which implies that $C_{L}\subset C_{A}$
is also finite and that neither $C_{A}$ nor $C_{L}$ contains~$0$.
By Proposition~\ref{prop:finite-all}, there exists a matroid~$M_{1}$
such that $\chi_{L}(M_{1})=C_{L}$ and $\chi_{A}(M_{1})=\mathbb{P}$.
By Proposition~\ref{prop:finite-finite}, there exists another matroid~$M_{2}$
such that $\chi_{L}(M_{2})=\chi_{A}(M_{2})=C_{A}$. Since
the characteristic set of a direct sum is the intersection of the
characteristic sets, $\chi_{L}(M_{1}\oplus M_{2})=C_{L}$ and $\chi_{A}(M_{1}\oplus M_{2})=C_{A}$.

Second, suppose that $C_{A}$ is cofinite, but $C_L$ is finite, which
again implies that $0$ is not in $C_L$ and~$C_A$.
Then by Proposition~\ref{prop:cofinite-cofinite}, there
exists a matroid $M_{1}$ such that $\chi_{L}(M_{1})=\chi_{A}(M_{1})=C_{A}\cup\{0\}$.
By Proposition~\ref{prop:finite-all}, there exists a matroid $M_{2}$ such that
$\chi_{L}(M_{2})=C_{L}$ and $\chi_{A}(M_{2})=\mathbb{P}$. Again, the
characteristic sets of a direct sum are the intersections of the characteristic
sets, so $\chi_L(M_1 \oplus M_2) = C_L$ and $\chi_A(M_1 \oplus M_2) = C_A$.

Third, suppose that $C_A$ and $C_L$ are both cofinite, which implies that $0 \in
C_L \subset C_A$. Similarly, we use Proposition~\ref{prop:cofinite-cofinite} to
construct a matroid $M_1$ such that $\chi_L(M_1) = \chi_A(M_1) = C_A$. Moreover,
by Theorem~\ref{thm:frob-flock-all}, whose proof doesn't use anything in this
section, $\chi_{F}(M_{1})=\mathbb{P}$. By Proposition~\ref{prop:cofinite-all},
there exists a matroid $M_{2}$ such that $\chi_{L}(M_{2})=C_{L}$ and
$\chi_{A}(M_{2})=\mathbb{P}\cup\{0\}$, and consequently $\chi_F(M_2) = \mathbb
P$. Because characteristic sets of direct
sums are the intersections of characteristic sets, $\chi_L(M_1 \oplus M_2) =
C_L$ and $\chi_A(M_1 \oplus M_2) = C_A$. The same is true for Frobenius flock
characteristic sets, by Theorems~4.11, 4.13, and~4.18 in\cite{Bollen}, so
$\chi_F(M_1 \oplus M_2) = \mathbb P$.
\end{proof}

\section{Infinite algebraic characteristic sets}
\label{sec:Infinite-algebraic-char}

The following proposition gives explicit examples of algebraic
characteristic set which are neither finite nor cofinite. Our construction works
similarly to Example 2 in \cite{EvansHrushovski}.

\begin{prop}\label{prop:inf-nocofinite}
Let $n$ be a positive integer. Then there exists a matroid $M_n$ such that 
\[\chi_{A}(M_n)=\left\{ p \in \PP:p\not\equiv1\bmod n\right\}.\]
\end{prop}

\begin{proof}
Let $k$ be the least integer such that $m = kn$ is greater than~$6$.
We define a new system of equations $\Phi_{n}$ satisfying the conditions
in Lemma \ref{lem:evans-hur}, in terms of the variables $x_{0},x_{1},y_{1},\cdots,$ $y_{m-1},z_{1},z_{2},z_{3}$
by the following equations:
\begin{align*}
x_{0}&=0  & y_2 &= y_1 \cdot y_1 & z_2 &=y_k \cdot z_1 \\
x_{1}&=1 & y_3 &= y_2 \cdot y_1 & z_3 &= z_1 \cdot y_k \\
&&& \;\;\vdots \\
&& y_{m-1} &= y_{m-2} \cdot y_1 \\
&& x_1 &= y_{m-1} \cdot y_1
\end{align*}

Now, use Lemma~\ref{lem:evans-hur} to construct a matroid~$M_n$ from the
equations $\Phi_n$. Any algebraic realization of $M_n$ will yield a solution to these
equations in the division ring of the endomorphism ring of a 1-dimensional
group over~$K$. This solution must satisfy:
\begin{align*}
x_{0} & =  0         & y_2 &= y_1^2 & z_2 &= y_kz_1 \\
x_{1} & =  1 = y_1^m & y_3 &= y_1^3 & z_3 &= z_1 y_k \\
 &&& \;\;\vdots\\
&&y_{m-1} & =  y_1^{m-1}
\end{align*}
Because the $y_i$'s are distinct from $1$ in this solution, $y_1$ must be a
primitive $m$th root of unity. Also, in order for $z_2 = y_k z_1$ and $z_3 =
z_1y_k$ to be distinct, the division ring must be non-commutative, which implies
that $0 \notin \chi_A(M_n)$ and in positive characteristic, a solution must come
from a non-commutative endomorphism ring of an elliptic curve or~$G_m$.

In the former case, the endomorphism ring is contained in a quaternion algebra
over~$\QQ$, all of whose elements have degree at most $2$ over $\QQ$. On the
other hand, a primitive $m$th root of unity for $m > 6$ has degree at least $3$
over $\QQ$, so the elliptic curve case is not possible.

Therefore, an algebraic realization corresponds to a solution to $\Phi_{n}$ in
the division ring~$Q$ of the ring of $p$-polynomials. The element $y_k = y_1^k$
is a primitive $n$th root of unit. If $p \equiv 1 \bmod{n}$, then the polynomial
$t^n -1$ splits in $\FF_p$, and so $y_k$ is contained in $\FF_p$. However, any
element of $\FF_p$ is in center of $Q$, which contradicts the inequality between
$z_2 = y_k z_1$ and $z_3 = z_1 y_k$.
Therefore, if $M_p$ is algebraically
realizable over a field, the field most have positive characteristic $p
\not\equiv 1 \bmod{n}$.

Conversely, suppose that $p \not\equiv 1 \bmod{n}$, and we will construct a
solution to $\Phi_n$ in $\overline \FF_p[F]$. We choose $y_1$ to be a primitive
$m$th root of unit in $\overline\FF_p$ and set $y_i = y_1^k$. In particular,
$y_k$ is a primitive $n$th root of unity, which is not contained in $\FF_p$
because $p \not\equiv 1 \bmod{n}$. We set $z_1 = F$, so that $z_2 = y_k F$ and
$z_3 = F y_k = y_k^pF$ are distinct because $y_k \notin \FF_p$. Thus, $M_p$ is
algebraically realizable over $\overline\FF_p$.
\end{proof}

The proof of Theorem~\ref{thm:char-density} uses the following elementary lemma
from analysis, whose proof we include for the convenience of the reader.

\begin{lem}\label{lem:density}
Let $(x_n)$ be a sequence of positive numbers such that $x_n \rightarrow 0$
as $n \rightarrow \infty$ but $\sum_{n=1}^{\infty} x_n =
\infty$. Then for any $a, \delta > 0$, there exists a finite set of integers
$A$ such that $a - \delta < \sum_{n\in A} x_n < a + \delta$. 
\end{lem}

\begin{proof}
Let $N$ be such that $x_n<2 \delta$ for all $n\geq N$.
Let $M\geq N$ be the minimal index such that $\sum_{n=N}^{M} x_n > a -
\delta$, which exists since $\sum_{n=N}^{\infty}
x_n=\infty$. Then, by minimality, $\sum_{n=N}^{M-1} x_n \leq a- \delta$, so
\[
\sum_{n=N}^{M} x_n = \sum_{n=N}^{M-1} x_n + x_M < (a-\delta) + 2 \delta <
a+\delta. \]
    Thus $A = \{ N,N+1,\ldots,M \}$ is a set as in the lemma statement.
\end{proof}

\begin{proof}[\textbf{Proof of Theorem \ref{thm:char-density}}]
Let $q$ be a fixed prime and $M_q$ be the matroid obtained from the Proposition
\ref{prop:inf-nocofinite} with $\chi_A(M_q)=\left\{ p\text{ prime
}:p\not\equiv1\bmod q\right\}$. By Dirichlet's theorem on arithmetic
progressions, the set of primes $p$ such that $p \equiv 1\bmod q$ has natural
density $1/(q-1)$ and therefore, $\chi_A(M_q)$ has natural density $(q-2)/(q-1)$.

More generally, for any finite set $S$ of primes, the algebraic characteristic
set of the direct sum $\bigoplus_{q \in S} M_q$ is the set of primes~$p$ such
that $p \not\equiv 1 \bmod{q}$ for all $q$ in~$S$. By the Chinese Remainder
Theorem, there are $\prod_{q \in S} (q-2)$ non-zero congruence classes modulo
$\prod_{q \in S} q$ which satisfy these congruence inequalities for all $q \in
S$. Therefore, by Dirichlet's theorem on arithmetic progression, the natural
density of $\chi_A(\bigoplus_{q \in S} M_q))$ is $\prod_{q \in S} (q-2)/(q-1)$.

Now, we proceed to find a suitable set $S$.
We let $q_n$ denote the $n$th prime, and set
\[x_n = -\log \left( \frac{q_n-2}{q_n-1} \right)=-\log \left( 1 - \frac{1}{q_n-1}
\right) \geq \frac{1}{q_n - 1} \geq \frac{1}{q_n}.\]
Then $x_n \rightarrow 0$ since $q_n \rightarrow \infty$, and since
$\sum_{n=1}^\infty 1/q_n$ diverges, so does $\sum_{n=1}^\infty x_n$.
Therefore, Lemma~\ref{lem:density} with $a = -\log \alpha$ and $\delta =  \log
(\alpha + \epsilon) - \log(\alpha)$ gives us a finite set $A$ such that
\[\left\lvert a - \sum_{n \in A} -\log \frac{q_n-2}{q_n-1} \right\rvert < \delta.\]
Because $\log$ is a concave function, $\delta < \log(\alpha) - \log(\alpha -
\epsilon)$, which implies
\[\left\lvert \alpha - \prod_{n \in A} \frac{q_n-2}{q_n-1} \right\rvert < \epsilon\]
Then, consider $M = \bigoplus_{n \in A}{M_{q_n}}$, and we have shown that the
density of $\chi_A(M)$ is $\prod_{n \in A} (q_n-2)/(q_n-1)$,
 and so $\lvert d\left(\chi_A(M) \right) - \alpha\rvert < \epsilon$.
\end{proof}

\section{Stretching Frobenius flocks \label{sec:Constructing-Frobenius-flocks}}

In this section, we prove Theorem~\ref{thm:frob-flock-all}, establishing the
existence of Frobenius flocks for any matroid which is linear over a field of
characteristic~0 and Theorem~\ref{thm:frob-dual}, proving that the Frobenius
flock representability of a matroid is closed under duality. Both results use
the same technical tool, which is a way of stretching linear flocks, which are a
more general object than Frobenius flocks.

For the definition of the linear flock, we need notations and definitions
of deletion and contraction for vector spaces. Let $E$ be a finite set and $K$ a field.
For $v\in K^{E}$, and $I\subseteq E$, define
$v_{I}\in K^{I}$ be the restriction of $v$ to the coordinates indexed by $I$ and
for a linear subspace $V\subseteq K^{E}$ and $I\subseteq E$ define
deletion and contraction to be
\[
V\setminus I=\left\{ v_{E\setminus I}\mid v\in V\right\} \text{ and } V/I=\ensuremath{\left\{ v_{E\setminus I}\mid v\in V,v_{I}=0\right\} } , 
\]
respectively, 
both of which are subspaces of $K^{E - I}$. Since $V\setminus I$ is the projection
of $V$ to $K^I$, and $V/(E - I)$ is the kernel of that projection, the
rank-nullity theorem implies that $\dim V
\setminus I + \dim V/(E-I) = \dim V$. It is also easy to see that when applied
to disjoint sets, deletion, and contraction commute with each other, and also
that multiple deletions or contractions can be combined.

Each vector space $V \subseteq K^E$ defines a matroid whose bases are the sets $B$ such that $V \setminus (E - B) = K^{B}$. We denote it by $M(V)$. The deletion and contraction of vector spaces are closely related to deletion and contraction of matroids. For instance, for any $I\subseteq E$, $M(V/I)=M/I$ and $M(V\setminus I)=M\setminus I$.

Now suppose that $\phi$ is an automorphism of $K$. Then for any $v\in K^{E}$
we can define an action of $\phi$ coordinate-wise:
\[
\phi v=\left(\phi\left(v_{i}\right)\right){}_{i\in E}
\]
and for a vector space $V\subseteq K^{E}$, we have $\phi V=\left\{ \phi v\mid
v\in V\right\} $, which is also a vector space.

Following \cite[Def.~4.1]{Bollen}, a $\phi$\emph{-linear flock} of $E$ over $K$ is defined to be a map $V_{\bullet}\colon\alpha \mapsto V_{\alpha}$
which assigns a $d$-dimensional linear subspace $V_{\alpha}\subseteq K^{E}$ to
each $\alpha\in\mathbb{Z}^{E}$, such that :
\begin{enumerate}
\item[(LF1)] $V_{\alpha}/i=V_{\alpha+e_{i}}\setminus i$ for all $\alpha\in\mathbb{Z}^{E}$
and $i\in E$; and 
\item[(LF2)] $V_{\alpha+\mathbf{1}}=\phi V_{\alpha}$ for all $\alpha\in\mathbb{Z}^{E}$.
\end{enumerate}
\noindent Here $e_{i}$ is the $i$th unit vector in $\ZZ^n$ and $\mathbf{1}\in
\ZZ^n$ is the vector whose entries are all $1$.
If $\phi=F^{-1}$, where
$F\colon x\rightarrow x^{p}$ is the Frobenius map, then we call a $F^{-1}$-linear flock a
\emph{Frobenius flock} \cite[Sec.~4]{BollenDraismaPendavingh}.

For each $\alpha \in \ZZ^n$, the vector space $V_\alpha$ defines a matroid
$M(V_\alpha)$ whose bases are the $d$-element sets $B$ such that $V \setminus (E
- B) = K^{B}$. The union of these sets of bases, for all $\alpha \in \ZZ^n$, is
also a matroid, which we call the support matroid of $V_\alpha$
\cite[Lem.~17]{BollenDraismaPendavingh}. Let $M$ be a matroid. If there exists a
Frobenius flock $V_\bullet$ with support matroid $M$, then  $V_\bullet$ is a
\emph{Frobenius flock representation} of~$M$.

We now establish a lemma allowing us to stretch Frobenius flock
representations:
\begin{lem}
\label{lem:123}Let $V_{\bullet}$ be a $\phi$-linear flock over a field $K$.
Suppose that $\psi$ is an automorphism of $K$ such
that $\psi^{m}=\phi$. Then, there exists a $\psi$-linear flock $V_{\bullet}'$
where $V_{m\alpha}'=V_{\alpha}$ for all $\alpha\in\mathbb{Z}^{n}$, and whose
support matroid is the same as the support matroid of $V_{\bullet}$.
\end{lem}

\begin{proof}
Let $\beta\in\mathbb{Z}^{n}$, and write $\beta=m\alpha+(r_{1},\ldots,r_{n})$
where $0\leq r_{i}<m$ and $\alpha\in\mathbb{Z}^{n}$. For any $0 \leq k < m$, we define the sets
$I_{<k}=\left\{ i:r_{i}<k\right\} $, $I_{>k}=\left\{ i:r_{i}>k\right\} $
and $I_{k}=\left\{ i:r_{i}=k\right\} $.

Now let us define the $K$-vector space
\[
V_{\beta}'=\bigoplus_{k=0}^{m-1}\psi^{k}V_{\alpha}/I_{>k}\setminus I_{<k},
\]
and we claim that as $\beta$ ranges over all elements of $\ZZ^n$, $V_{\bullet}'$
defines a $\psi$-linear flock. Note that a term $\phi^k
V_{\alpha}/I_{>k}\setminus I_{<k}$ in the definition of $V_{\beta}'$ is a
subspace of $K^{I_{k}}$ and so the direct sum gives a vector subspace of $K^E$
via the isomorphism $K^E \cong\oplus_{k=0}^{m-1} K^{I_{k}}$. Also, if $\beta = m
\alpha$, meaning that $r_i = 0$ for all $i$, then only the $k=0$ summand of the
definition of $V_\beta'$ is non-trivial, and this shows that $V_{m\alpha}' =
V_{\alpha}$.

As noted above, the rank-nullity theorem implies that 
\[d = \dim V_\alpha = \dim V / I_{> 0} + \dim V \setminus I_{0}.\]
By induction, and because the sets $I_k$
partition $E$, $d = \sum_{k=0}^{m-1} \dim V_\alpha / I_{>k} \setminus I_{<k}$,
which implies that $\dim V_{\beta}' = d$.

We check the axiom (LF2) of a linear flock first. Consider
\[
\beta+\mathbf{1}=m\alpha'+(r_{1}',\ldots,r_{n}')
\]
and if we define $\alpha'=\alpha+e_{I_{m-1}}$, $I'_{j}=\left\{ i:r'_{i}=j\right\} =I_{j-1}$
for $1\leq j\leq m-1$ and $I'_{0}=I_{m-1}$, then similarly to the decomposition
$\beta' = m \alpha' + (r_1', \ldots, r_n')$, where $r_i = j$ if and
only if $i \in I_j'$. In addition, we also define $I_{<k}' = \bigcup_{j < k}
I_{j}'$ and $I_{>k}' = \bigcup_{j > k} I_{j}'$, which means that $I_{<k}' =
I_{<k-1} \cup I_{m-1}$ and $I_{>k}' = I_{>k-1} - I_{m-1}$, where $-$ denotes the
set difference, to distinguish it from matroid deletion.

For $I\subseteq E$, the following generalization holds in analogy
with Lemma 9 of \cite{BollenDraismaPendavingh}, 
\begin{lyxlist}{00.00.0000}
\item [{(LF1')}] $V_{\alpha}/I=V_{\alpha+e_{I}}\setminus I\text{ for all }\alpha\in\mathbb{Z}^{n}\text{ and }I\subseteq\left\{ 1,2,\ldots,n\right\} $
where $e_{I}=\sum_{i\in I}e_{i}$. 
\end{lyxlist}

Then, we have
\begin{align*}
V_{\beta+\mathbf{1}}'  = &  \bigoplus_{k=0}^{m-1}\psi^{k}V_{\alpha'}/I'_{>k}\setminus I'_{<k}
&& \mbox{by definition of $V_{\bullet}'$}\\
  = & \left(\bigoplus_{k=1}^{m-1}\psi^{k}V_{\alpha+e_{I_{m-1}}}/\left(I_{>k-1} - I_{m-1}\right)\setminus\left(I_{<k-1}\cup I_{m-1}\right)\right)\\
 & \oplus V_{\alpha+e_{I_{m-1}}}/I_{<m-1}
&&\mbox{by the above identities }\\
  = & \left(\bigoplus_{k=1}^{m-1}\psi^{k}V_{\alpha}/I_{>k-1}\setminus I_{<k-1}\right)\oplus V_{\alpha+\mathbf{1}}\setminus I_{<m-1}
&&\mbox{by (LF1)} \\
  = & \left(\bigoplus_{k=1}^{m-1}\psi^{k}V_{\alpha}/I_{>k-1}\setminus I_{<k-1}\right)\oplus\phi V_{\alpha}\setminus I_{<m-1}
&&\mbox{by (LF2)} \\
  = & \left(\bigoplus_{k=1}^{m-1}\psi^{k}V_{\alpha}/I_{>k-1}\setminus I_{<k-1}\right)\oplus\psi\cdot\psi^{m-1}V_{\alpha}\setminus I_{<m-1}
&&\mbox{because $\phi = \psi^m$} \\
  = & \psi V_{\beta}'
&&\mbox{by definition of $V_{\beta}'$}
\end{align*}
This completes the proof of (LF2).

Now we consider the axiom (LF1), which says that $V_\beta/i =
V_{\beta+e_i} \setminus i$. We first consider the case when $i \notin
I_{m-1}$ and let $j = r_i$, so that $i \in I_j$. Therefore, the vector
$\beta+e_i$ can be written as $m \alpha + (r_1', \ldots, r_n')$, where $r_1',
\ldots, r_n' < m$ and $r_k' = r_k$ unless $k=i$ in which case $r_i' = r_i + 1$.
Then, if $I_{<k}' = \{i : r_i' < k\}$ and $I_{>k}'= \{i : r_i' > k\}$, as usual,
then $I_{<j+1}' = I_{<j+1} - \{i\}$ and $I_{>j}' =
I_{>j} \cup \{i\}$, but other than these two exceptions, $I_{<k}' = I_{<k}$ and
$I_{>k}' = I_{>k}$. Therefore, the definition of $V_{\bullet}'$ gives us:

\begin{align*}
V_{\beta+e_{i}}'=&\left(\bigoplus_{{\substack{k=0\\
k\neq j,j+1
}
}}^{m-1}\psi^{k}V_{\alpha}/I_{>k}\setminus
I_{<k}\right)
\oplus\big(\psi^{j}V_{\alpha}/\left(I_{>j}\cup\left\{ i\right\}
\right)\setminus I_{<j}\big)\\
&\oplus\big(\psi^{j+1}V_{\alpha}/I_{>j+1}\setminus\left(I_{<j+1} - \left\{
i\right\} \right) \big).
\end{align*}

The deletion of the $i$th component only affects the summand contained in
$K^{E_j}$, which is the last summand, so by combining the deletions:
\begin{align*}
V_{\beta+e_{i}}'\setminus i =&\left(\bigoplus_{{\substack{k=0\\
k\neq j,j+1
}
}}^{m-1}\psi^{k}V_{\alpha}/I_{>k}\setminus I_{<k}\right)
\oplus\big(\psi^{j}V_{\alpha}/\left(I_{>j}\cup\left\{ i\right\} \right)\setminus
I_{<j} \big)\\
&\oplus\big(\psi^{j+1}V_{\alpha}/I_{>j+1}\setminus I_{<j+1} \big) \\
=&\left(\bigoplus_{{\substack{k=0\\ k\neq j}}}^{m-1} \psi^k V_\alpha/I_{>k}
\setminus I_{<k} \right)
\oplus \big( \psi_J V_{\alpha}/(I_{>j} \cup \{i\}) \setminus I_{<j} \big) \\
=&\left(\bigoplus_{k=0}^{m-1} \psi^k V_\alpha/I_{>k}
\setminus I_{<k} \right) / \{i\} = V_{\beta}' /\{i\},
\end{align*}
because the contraction of $\{i\}$ only affects the $k=j$ summand. 
This completes the proof of~(LF1) when $i \not\in I_{m-1}$.

Now suppose that $i\in I_{m-1}$. In this case
$\beta+e_{i}=m\alpha'+(r_{1}',\ldots,r_{n}')$ where $\alpha'=\alpha+e_{i}$ ,
$I'_{k}=\left\{ i:r'_{i}=k\right\} =I_{k}$
for $k\neq0,m-1$, $I'_{m-1}=I_{m-1}\setminus\left\{ i\right\} $
and $I'_{0}=I_{0}\cup\left\{ i\right\} $. Then, 
\begin{eqnarray*}
V_{\beta+e_{i}}' & = &
\left(\bigoplus_{k=1}^{m-2}\psi^{k}V_{\alpha+e_{i}}/\left(I_{>k} - \left\{ i\right\} \right)\setminus\left(I_{<k}\cup\left\{ i\right\} \right)\right)\oplus
\big(V_{\alpha+e_{i}}/\left(I_{>0} - \left\{ i\right\} \right)\big)\\
 & & \oplus\big(\psi^{m-1}V_{\alpha+e_{i}}\setminus\left(I_{<m-1}\cup\left\{
i\right\} \right)\big)\\
V_{\beta+e_{i}}'\setminus i & = & \left(\bigoplus_{k=1}^{m-2}\psi^{k}V_{\alpha+e_{i}}/\left(I_{>k}\setminus\left\{ i\right\} \right)\setminus\left(I_{<k}\cup\left\{ i\right\} \right)\right)
\oplus \big(V_{\alpha+e_{i}}/\left(I_{>0} - \left\{ i\right\}) \setminus \{i\} \right)\big)\\
 &  & \oplus\big(\psi^{m-1}V_{\alpha+e_{i}}\setminus\left(I_{<m-1}\cup\left\{
i\right\} \right)\big)\\
% \end{eqnarray*}
% \begin{eqnarray*}
V_{\beta+e_{i}}'\setminus i & = & \left(\bigoplus_{k=1}^{m-2}\psi^{k}V_{\alpha}/I_{>k}\setminus
 I_{<k}\right)\oplus \big(V_{\alpha}/I_{>0}\big)\oplus\big(\psi^{m-1}V_{\alpha}\setminus
 I_{<m-1} / \{i\}\big)\\
  & &\hspace{5.5cm} \left(\mbox{by (LF1) in each summand}\right)\\
& = & \left(\bigoplus_{k=0}^{m-1}\psi^{k}V_{\alpha}/I_{>k}\setminus I_{<k}\right)/\{i\}\\
& = & V_{\beta}'/i,
\end{eqnarray*}
which completes the proof of (LF1) and thus that $V_{\bullet}'$ is a matroid flock.

Finally, we want to show that the support matroids of $V_{\bullet}$ and $V_{\bullet}'$
are the same. Since $V_{m \alpha}' = V_{\alpha}$, any basis of the support
matroid of $V_{\bullet}$ will also be a basis of the support matroid of
$V_{\bullet}'$. For the converse, we suppose that
$\beta$ is any coordinate in $\ZZ^n$ and $B$ is any subset of~$E$. Then, with
$\alpha$, $I_{>k}$, and $I_{<k}$ as before,
\[ V_{\beta}' = \bigoplus_{k=0}^{m-1} \psi^k V_{\alpha}/I_{>k} \setminus I_{<k}
\]
and
\[ V_{\beta}' / (E - B) = \bigoplus_{k=0}^{m-1} \psi^k V_{\alpha}/I_{>k} / (I_k
- B) \setminus I_{<k}. \]

The deletion of a vector space always contains the contraction of the same set,
and thus, \begin{align*} V_{\beta}' / (E-B) &\subset \bigoplus_{k=0}^{m-1} \psi^k V_{\alpha} / I_{>k} / (I_k -
B) / (I_{<k} - B) \setminus (I_{<k} \cap B) \\
&= \bigoplus_{k=0}^{m-1} \psi^k V_{\alpha} / (E -B) / (I_{>k} \cap B) \setminus
(I_{<k}
\cap B).
\end{align*}
However, this last expression is the same construction that was used to make
$V_\beta'$, but applied to $V_{\alpha} / (E-B)$. Therefore, its vector space
dimension is the same as that of $V_{\alpha} / (E-B)$, which, by the
containment, implies that $\dim
V_\beta' / (E-B) \leq \dim V_{\alpha} / (E-B)$. If $B$ is a basis of the support
matroid of $V_{\bullet}'$, then $\dim V_{\beta}' = \lvert B \rvert$, which means
that $B$ is also a basis of the support matroid of~$V_{\bullet}$. This concludes
the proof
that $V_\bullet$ and $V_\bullet'$ have the same support matroids.
\end{proof}

Using Lemma \ref{lem:123}, we now prove Theorem \ref{thm:frob-flock-all} and
Theorem \ref{thm:frob-dual}. \begin{proof}[\textbf{Proof of Theorem
\ref{thm:frob-flock-all}}] By \cite{Ingleton1971}, if $0\in\chi_{L}(M)$, then
$M$ has a representation over a finite extension of the rationals, i.e.\ a number
field~$K$. Let $\mathcal{O}_{K}$ be the ring of integers in the number
field~$K$. Then, using going-up theorem~\cite[Thm.~20]{Marcus}, for any prime
$p$, there exists a prime ideal $\mathfrak{P}\subset\mathcal{O}_{K}$ such that
$\mathfrak{P}\cap\mathbb{Z}=\left(p\right)$. By \cite[Thm. 14]{Marcus},
$\mathcal{O}_{K}$ is a Dedekind domain and if $\mathfrak{I}$ is any non-zero
ideal in $\mathcal{O}_{K}$, then $\mathcal{O}_{K}/\mathfrak{I}$ is finite. So
$\mathfrak{P}$ is a maximal ideal and $\mathcal{O}_{K}/\mathfrak{P}$ is a finite
field. The containment of $\mathbb{Z}$ in $\mathcal{O}_{K}$ induces a
ring-homomorphism $\mathbb{Z}\rightarrow\mathcal{O}_{K}/\mathfrak{P}$, and the
kernel is $\mathfrak{P}\cap\mathbb{Z}=\left(p\right)$. So, we obtain an
embedding $\mathbb{F}_{p}\rightarrow\mathcal{O}_{K}/\mathfrak{P}$. Then
$\mathcal{O}_{K}/\mathfrak{P}$ is an extension of finite degree over
$\mathbb{F}_{p}$. Thus, $\mathcal{O}_{K}/\mathfrak{P}\cong\mathbb{F}_{p^m}$
for some $n$. Also, any localization of a Dedekind domain at a non-zero prime
ideal is a discrete valuation ring  \cite[Thm. 15, Ch. 16]{DummitFoote}. So,
there exists a valuation $\nu:K^{*}\rightarrow\mathbb{Z}$ whose valuation ring
has residue field isomorphic to~$\FF_{p^m}$. Using this valuation with
\cite[Lem.~3.5]{BollenCartwrightDraisma}, we can construct a linear flock with trivial automorphism over a finite field
$\mathbb{F}_{p^m}$.

Now consider the inverse Frobenius automorphism $F^{-1}\colon x\mapsto x^{-p}$
of $\mathbb{F}_{p^m}$, whose iteration $F^{-m}$ is the trivial automorphism.
Then, using Lemma \ref{lem:123} with $\psi=F^{-1}$,
we have $M$ has a Frobenius flock representation over a field of
characteristic $p$. Therefore, $\chi_{F}(M)=\mathbb{P}$. 
\end{proof}

\begin{proof}[\textbf{Proof of Theorem \ref{thm:frob-dual}}]
Let $V_\bullet \colon \alpha \mapsto V_{\alpha}$ be a Frobenius flock representation over $K$ of $M$. Then the
dual of $V_\bullet$ is a defined as $V^*_\bullet \colon \alpha \mapsto V^\perp_{-\alpha}$
\cite[Def. 4.15]{Bollen}, which form a $F$-linear flock over
$K$ with support matroid $M^*$~\cite[Thm.~4.16]{Bollen}.
By~\cite[Thm.~4.30]{Bollen}, the flock $V^*_\bullet$ is determined by its
skeleton, which is a finite number of matroid realizations, together with a
finite number of
compatibility conditions, all of which are algebraic. Therefore, we can assume
that the skeleton is defined over a finite extension of $\FF_p$, and thus that
$V^*_\bullet$ is a $F$-linear flock over a finite extension of~$\FF_p$.

Since the Frobenius endomorphism $F$ has finite order in a finite extension of
$\FF_p$, then $F^{-m} = F$ for some $m$. Then, using Lemma \ref{lem:123}, we get
a $F^{-1}$-linear flock with the support matroid $M^*$. Therefore, $M^*$ has a
Frobenius flock representation over $K$ which implies  $\chi_F(M) \subset
\chi_F(M^*)$. Furthermore, since $ (M^*) ^* = M$,  we have the inclusion in the
other direction, and so $\chi_F(M^*)
= \chi_F(M)$. 
\end{proof}

\section{Finite Frobenius flock characteristic sets \label{sec:Brylawski-matroids}}

In this section, we give an examples of matroids with finite, non-singleton
Frobenius flock characteristic set, based on the construction
in~\cite{Gordon}.

\begin{defn}
\label{def:bry}Consider a set of primes $\{p_{1},p_{2},\ldots,p_{k}\}$
and let $n=p_{1}\cdots p_{k}+1$ and $s=\lfloor\log_{2}n\rfloor$.
For $0\leq i\leq s$, set $b_{i}=\lfloor n/2^{(s-i+1)}\rfloor.$ Then
$b_{0}=0,b_{1}=1$, $b_{2}=2$ or $3$ and in general, \textbf{$b_{i}=2b_{i-1}$
}or $2b_{i-1}+1.$ The \emph{Brylawski matrix} $N_{n}$, introduced in a slightly
different form
in~\cite{Brylawski1982}, is the $3 \times (2s+6)$
matrix: 
\begin{align*}
\begin{blockarray}{cccccccccccccc}
v_{1} & v_{2}& v_{3} & v_{4} & v_{5} & v_{6} & w_{1}& u_{1}& \cdots &
w_i & u_i & \cdots & w_s & u_s\\
\begin{block}{(cccccccccccccc)}
1 & 0 & 0 & 1 & 1 & 1 & 1 & 0 &  & 1 & 0 &  & 1 & 0\\
0 & 1 & 0 & 1 & 1 & 0 & 2 & 1 & \cdots & 2 & 1 & \cdots & 2 & 1\\
0 & 0 & 1 & 1 & 0 & 1 & 1 & 1 &  & b_{i} & b_{i} &  & b_{s} & b_{s} \\
\end{block}
\end{blockarray}
\end{align*}We call the set of primes $\left\{ p_{1},p_{2},\ldots,p_{k}\right\} $
a \emph{Gordon-Brylawski set}, if for each pair of indices $0 \leq i < j \leq
s$ such that $(i,j) \not\in \{ (0,1), (1,2)\}$,
and each prime $p_i$, the difference $b_j - b_i$ is not $0$ or $\pm 1$ modulo
$p_i$~\cite{Gordon}.
\end{defn}

Note that the $(0,1)$ and possibly $(1,2)$ are the only pairs of indices $(i,j)$
such that $b_j - b_i$ can equal $1$. Since these differences will therefore be
$1$ uniformly modulo all primes $p_i$, this does not cause a problem for the
construction.

\begin{ex}\label{ex:gb-80}
A computation shows that the 80 consecutive primes beginning with 12811987 form
a Gordon-Brylawski set.
\end{ex}

The following proposition proves Theorem~\ref{thm:finite-flock-char}:

\begin{prop}
\label{prop:bry}Let $N_{n}$ be the Brylawski matrix where $n=p_{1}\cdots p_{k}+1$
and $M_{n}$ be the matroid which is linearly represented over $\mathbb{F}_{p_{1}}$
by the matrix $N_{n}$. Then, $\chi_{F}(M_{n})\subseteq\left\{ p_{1},\ldots,p_{k}\right\} $.
If $\left\{ p_{1},\ldots,p_{k}\right\} $ is a Gordon-Brylawski
set, then $\chi_{F}(M_{n})=\left\{ p_{1},\ldots,p_{k}\right\} $. 
\end{prop}

\begin{proof}
Assume that $M_{n}$ has a Frobenius flock representation over a field
$K$ of characteristic~$p$. Let $A$ be the restriction of $M_n$ to its first 4
elements, corresponding to first for columns of~$N_n$.
Then $A$ is isomorphic to $U_{3,4}$, which is
rigid~\cite[Lem.~53]{BollenDraismaPendavingh}, which means that any valuations
on the matroid $U_{3,4}$ is projectively equivalent to the valuation which is
constant~$0$. Then by \cite[Lem.~46]{BollenDraismaPendavingh}, there exist
$\alpha\in\mathbb{Z}^{E}$ such that $M_{\alpha}=M\left(V_{\alpha}\right)$
contains the elements of $A$ as a circuit. Now we want to show that $V_{\alpha}$
equals the row space of the Brylawski matrix $N_{n}$ over~$\mathbb{F}_{p}$.  Let
$B$ be the matrix representing $V_\alpha$ and $v_1', \ldots v_6', w_1', u_1',
\ldots, w_s', u_s'$ denote the columns of $B$, analogous to the labeling of the
columns of the Brylawski matrix~$N_n$ in Definition~\ref{def:bry}.

As the first
four elements form a circuit, we may use row operations on the matrix $B$ in
such a way that the columns corresponding to this circuit are as in the matrix
$N_{n}$. Since $\left\{ v_{1},v_{2},v_{5}\right\} $ is a circuit, the third
entry in $v_{5}'$ is 0 and since $\left\{ v_{3},v_{4},v_{5}\right\} $ is a
circuit, then the two non-zero entries of $v_{5}'$ are the same. Then by the
column scaling we get that the $v_{5}'$ is the fifth column of the matrix
$N_{n}$. We can similarly conclude that, after scaling the columns $v_6'$ and
$u_1'$ are the same as the corresponding columns of $N_n$. 

The fact that $\{v_2, v_6, w_1\}$ is a circuit forces the first and last
entries of $w_1$ to be the same, and the circuit $\{v_5, u_1, w_1\}$ is a
circuit means the middle entry is the sum of the other two. Therefore, after
scaling, we can assume that $w_1'$ is the same as $w_1$. We now use induction
to show that for $i \geq 2$, $w_i'$ and $u_i'$ are the same as $w_i$ and $u_i$,
after scaling. Here, the circuit $\{v_3, w_1, w_i\}$ forces the first two
entries of $w_i$ to be $1$ and $2$ respectively, after scaling. Then, the one of
the minors
\[\begin{blockarray}{ccc}
v_1 & u_{i-1} & w_i \\
\begin{block}{|ccc|}
1 & 0 & 1 \\
0 & 1 & 2 \\
0 & b_{i-1} & b_i \\
\end{block}
\end{blockarray} = b_i - 2b_{i-1}
\quad\mbox{or}\quad
\begin{blockarray}{ccc}
v_6 & u_{i-1} & w_i \\
\begin{block}{|ccc|}
1 & 0 & 1 \\
0 & 1 & 2 \\
1 & b_{i-1} & b_i \\
\end{block}
\end{blockarray} = b_i - 2b_{i-1} -1
\]
is zero, depending on whether $b_i = 2b_{i-1}$ or $b_i = 2b_{i-1} + 1$ and then
this forces the last entry of $w_i'$ to be $b_i$. Finally, $u_i'$ is forced to
agree with $u_i$, after scaling, because of the circuits $\{v_5, w_i, u_i\}$ and
$\{v_2, v_3, u_i\}$. Therefore, we conclude that $V_\alpha$ is represented by
the Brylawski matrix $N_n$. We also have the minor 
\[\begin{blockarray}{ccc}
v_1 & w_1 & u_s \\
\begin{block}{|ccc|}
1 & 1 & 0 \\
0 & 2 & 1 \\
0 & 1 & b_s \\
\end{block}
\end{blockarray} =
2b_s -1 = n-1 = p_1 \ldots p_k,\]
which corresponds to a circuit of $M_\alpha$, and therefore forces the
characteristic $p$ to be one of $p_1, \ldots, p_k$. 

We have shown that $p=p_{i}$ for some $i$, so $\chi_{F}(M_{n})\subseteq\left\{ p_{1},\ldots,p_{k}\right\} $.
Since the set of primes $\left\{ p_{1},p_{2},\ldots,p_{k}\right\} $
is a Gordon-Brylawski set, then by \cite[Thm. 5]{Gordon}, we
have $\chi_{L}(M_{n})=\left\{ p_{1},\ldots,p_{k}\right\} $, therefore
$\chi_{F}(M_{n})=\left\{ p_{1},\ldots,p_{k}\right\} $. 
\end{proof}

\bibliographystyle{amsalpha}
\bibliography{project1027}

\providecommand{\bysame}{\leavevmode\hbox to3em{\hrulefill}\thinspace}
\providecommand{\MR}{\relax\ifhmode\unskip\space\fi MR }
\providecommand{\MRhref}[2]{%
  \href{http://www.ams.org/mathscinet-getitem?mr=#1}{#2}
}
\providecommand{\href}[2]{#2}
\begin{thebibliography}{BCD20}

\bibitem[BB19]{BakerBowler}
Matthew Baker and Nathan Bowler, \emph{Matroids over partial hyperstructures}, Adv. Math. \textbf{343} (2019), 821--863.

\bibitem[BCD20]{BollenCartwrightDraisma}
Guus~P. Bollen, Dustin Cartwright, and Jan Draisma, \emph{Matroids over one-dimensional groups}, Int. Math. Res. Not. IMRN (2020), rnaa175.

\bibitem[BDP18]{BollenDraismaPendavingh}
Guus~P. Bollen, Jan Draisma, and Rudi Pendavingh, \emph{Algebraic matroids and {Frobenius} flocks}, Adv. Math. \textbf{323} (2018), 688--719.

\bibitem[BK80]{BrylawskiKelly}
Tom Brylawski and D.~Kelly, \emph{Matroids and combinatorial geometries}, Carolina {Lecture} {Series}, University of North Carolina, Department of Mathematics, Chapel Hill, NC, 1980.

\bibitem[Bol18]{Bollen}
Guus~Pieter Bollen, \emph{Frobenius flocks and algebraicity of matroids}, Ph.D. thesis, Eindhoven University of Technology, 2018.

\bibitem[Bry82]{Brylawski1982}
Tom Brylawski, \emph{Finite prime-field characteristic sets for planar configurations}, Linear Algebra Appl. \textbf{46} (1982), 155--176.

\bibitem[DF99]{DummitFoote}
David Dummit and Richard~M. Foote, \emph{Abstract {Algebra}}, Prentice Hall, Upper Saddle River, N.J, 1999.

\bibitem[EH91]{EvansHrushovski}
David~M. Evans and Ehud Hrushovski, \emph{Projective planes in algebraically closed fields}, Proc. Lond. Math. Soc. \textbf{s3-62} (1991), no.~1, 1--24.

\bibitem[Gor88]{Gordon}
Gary Gordon, \emph{Algebraic characteristic sets of matroids}, J. Combin. Theory Ser. B \textbf{44} (1988), no.~1, 64--74.

\bibitem[Hum75]{Humphreys}
James~E. Humphreys, \emph{Linear {Algebraic} {Groups}}, Springer, New York, NY, 1975.

\bibitem[Ing71]{Ingleton1971}
Aubrey~W. Ingleton, \emph{Representation of matroids}, Combinatorial mathematics and its applications, vol.~23, London, 1971, pp.~149--167.

\bibitem[Kah82]{Kahn1982}
Jeff Kahn, \emph{Characteristic sets of matroids}, J. Lond. Math. Soc. \textbf{s2-26} (1982), no.~2, 207--217.

\bibitem[Lin85]{Lindstrom1985}
Bernt Lindstr{\"o}m, \emph{On the algebraic characteristic set for a class of matroids}, Proc. Amer. Math. Soc. \textbf{95} (1985), no.~1, 147.

\bibitem[Lin86]{Lindstrom1986}
Bernt Lindstr{\"o}m, \emph{A non-linear algebraic matroid with infinite characteristic set}, Discrete Math. \textbf{59} (1986), no.~3, 319--320.

\bibitem[Mar18]{Marcus}
Daniel Marcus, \emph{Number {Fields}}, Springer, New York, 2018.

\bibitem[Pen]{Pendavingh}
Rudi Pendavingh, \emph{Field extensions, derivations, and matroids over skew hyperfields}, preprint, available at \url{https://arxiv.org/abs/1802.02447}.

\bibitem[Rad57]{Rado1957}
R.~Rado, \emph{Note on independence functions}, Proc. Lond. Math. Soc. \textbf{s3-7} (1957), no.~1, 300--320.

\bibitem[Rei]{reid}
R.~Reid, \emph{Obstructions to representations of combinatorial geometries}, (unpublished, results appear as appendix to Chapter 24 of {\cite{BrylawskiKelly}}).

\bibitem[Sil86]{Silverman}
Joseph~H. Silverman, \emph{The {Arithmetic} of {Elliptic} {Curves}}, Springer, New York, NY, 1986.

\bibitem[Su23]{Su}
Ting Su, \emph{Matroids over skew tracts}, European J. Combin. \textbf{109} (2023), 103643.

\bibitem[Vam75]{Vamos1971}
P.~Vamos, \emph{A necessary and sufficient condition for a matroid to be linear}, M{\"o}bius Algebras (Proc. Conf. Univ. Waterloo, 1971) (1975), 166--173.

\bibitem[Whi35]{Whitney1935}
Hassler Whitney, \emph{On the abstract properties of linear dependence}, Amer. J. Math. \textbf{57} (1935), no.~3, 509.

\end{thebibliography}

\end{document}